\documentclass[psamsfonts]{amsart}
\usepackage{amssymb}
\usepackage{graphicx,color}
\usepackage{amssymb,amsfonts,amsthm}
\usepackage[all,arc]{xy}
\usepackage{enumerate}
\usepackage{mathrsfs}

\newtheorem{theorem}{Theorem}[section]
\newtheorem{lemma}[theorem]{Lemma}
\newtheorem{definition}[theorem]{Definition}
\newtheorem{remark}[theorem]{Remark}
\numberwithin{equation}{section}

\title[]{Interior Hessian estimates for Hessian quotient equations in dimension three}

\author{Heming Jiao}

\address{School of Mathematics and Institute for Advanced Study in Mathematics, Harbin Institute of Technology,
         Harbin, Heilongjiang 150001, China}
\email{jiao@hit.edu.cn}

\author{Zhenan Sui}

\address{Institute for Advanced Study in Mathematics, Harbin Institute of Technology, Harbin, Heilongjiang 150001, China}
\email{20170045@hit.edu.cn}

\thanks{The first author is supported by the National Natural Science Foundation of China (Grant No. 12271126). The second author is supported by the National Natural Science Foundation of China (Grant No. 12571212)}

\begin{document}

\begin{abstract}
In this paper, we establish the interior Hessian estimates for $2$-convex solutions to $\frac{\sigma_2}{\sigma_1} (D^2 u) = \psi (x,u)$ in dimension three.
In higher dimensions ($n \geq 4$), we prove the interior Hessian estimates for semi-convex solutions.
We provide a new method to prove the doubling inequality for smooth solutions in dimensions three and four. In higher dimensions ($n\geq 5$)
the doubling inequality is
proved under an additional dynamic semi-convexity condition which is the same to that in \cite{SY2025}. The method also applies to the equation $\sigma_2 (D^2 u) = \psi (x, u, \nabla u)$.

\end{abstract}

%\subjclass[2010]{Primary 53C21; Secondary 35J65, 58J32}

\maketitle

%\tableofcontents

\section {\large Introduction}

\vspace{4mm}

This paper is devoted to interior Hessian estimate for Hessian quotient equation
\begin{equation}
\label{Hessian}
F (D^2 u) := \bigg( \frac{\sigma_2}{\sigma_1} \bigg) (D^2 u) = \bigg( \frac{\sigma_2}{\sigma_1} \bigg) \Big( \lambda( D^2 u) \Big) = \psi (x, u),
\end{equation}
where $D^2 u$ is the Hessian of $u$, $\lambda = (\lambda_1, \cdots, \lambda_n)$ are the eigenvalues of $D^2 u$,
\[ \sigma_k (\lambda) = \sum_{1 \leq i_1 < \cdots < i_k \leq n} \lambda_{i_1} \cdots \lambda_{i_k} \]
is the $k$-th elementary symmetric function defined on $k$-th G\r arding's cone
\[ \Gamma_k = \big\{ \lambda \in \mathbb{R}^n \big\vert \sigma_j (\lambda) > 0, \ \ j = 1, \ldots, k \big\}, \]
and $\psi$ is a prescribed positive function defined on $B_1 (0) \times \mathbb{R}$. Here $B_1 (0)$ denotes the open ball centered at $0$ with radius $1$.

A function $u \in C^2$ is called $k$-convex if $\lambda(D^2 u) \in \Gamma_k$, for $k=1, \ldots ,n$.
Our main result is the following interior Hessian estimates.

\begin{theorem}
\label{thm1}
Let $u \in C^4 (B_1)$ be a $2$-convex solution to the equation \eqref{Hessian} in $B_1 \subset \mathbb{R}^3$ satisfying $\|u\|_{C^1(B_1(0))} \leq M < \infty$. Then
\begin{equation}
\label{interior}
|D^2 u (0)| \leq C,
\end{equation}
where $C$ is a positive constant depending only on $\|\psi\|_{C^{1,1}}$, $\|\frac{1}{\psi}\|_{L^\infty}$ and $\|u\|_{C^1 (B_1(0))}$.
\end{theorem}

In higher dimensions, we prove the interior Hessian estimates under an additional condition that the solution $u$ is semi-convex, namely,
there exists a positive constant $A$ such that $u + A|x|^2$ is convex.
\begin{theorem}
\label{thm2}
Let $u \in C^4 (B_1)$ be a $2$-convex solution to the equation \eqref{Hessian} in $B_1 \subset \mathbb{R}^n$ satisfying $\|u\|_{C^1(B_1(0))} \leq M < \infty$. If,
in addition, $u$ is semi-convex, then
\begin{equation}
\label{interior'}
|D^2 u (0)| \leq C,
\end{equation}
where $C$ is a positive constant depending only on $\|\psi\|_{C^{1,1}}$, $\|\frac{1}{\psi}\|_{L^\infty}$ and $\|u\|_{C^1 (B_1(0))}$.
\end{theorem}

Interior $C^2$ estimate is one of the fundamental estimates for fully nonlinear elliptic equations. One typical application is to find smooth solution on noncompact domains combined with Evans-Krylov interior $C^{2, \alpha}$ estimate. For $k$-Hessian equations
\begin{equation} \label{k_Hessian}
\sigma_k (D^2 u) = \psi(x, u),
\end{equation}
the counterexamples by Pogorelov \cite{Pogorelov1978} and Urbas \cite{Urbas1990} show that there is no interior $C^2$ estimate for \eqref{k_Hessian} when $k \geq 3$. When $k = 1$, \eqref{k_Hessian} reduces to semi-linear elliptic equation. When $k = 2$, interior $C^2$ estimate has become a longstanding problem.

The first result to $\sigma_2$ equation
\begin{equation} \label{2-equation}
\sigma_2 (D^2 u) = \psi(x, u, \nabla u)
\end{equation}
with $n = 2$ is given by Heinz \cite{Heinz1959}, using isothermal coordinates under Legendre-Lewy transform. In this special case, the equation is also of Monge-Amp\`ere type. Recent proofs for this special case can be found in Chen, Han and Ou \cite{CHO2016} using PDE method and Liu \cite{Liu2021} using the partial Legendre transform.

For $n \geq 3$, interior $C^2$ estimate to \eqref{2-equation} becomes extremely difficult, which calls for sophisticated analytic or geometric methods. When $\psi \equiv 1$, Warren and Yuan \cite{WY2009} derived interior $C^2$ estimate
for $n = 3$ via the minimal surface structure and a full strength Jacobi inequality.
For general $\psi$ and $n = 3$, interior $C^2$ estimate was obtained by Qiu \cite{Qiu2024-2} (see also \cite{Qiu2024} for prescribed scalar curvature equation). To be more concrete, in dimension $3$, Qiu \cite{Qiu2024-2} gave the first proof of a Jacobi inequality for $\ln \Delta u$:
\[ F^{ij} \nabla_{ij} \ln \Delta u \geq \epsilon F^{ij} (\ln \Delta u)_i (\ln \Delta u)_j, \]
and proposed the methodology of using maximum principle argument to prove a doubling inequality:
\[ \sup\limits_{B_1 (0)} \Delta u \leq C \Big( n, \Vert u \Vert_{C^1 \big( B_2 (0) \big)} \Big) \sup\limits_{B_{1/2}(0)} \Delta u. \]
Recently, Shankar and Yuan \cite{SY2025} proved the interior Hessian estimate for \eqref{2-equation} with $\psi \equiv 1$ and $n = 4$ by synthesizing the ideas of Qiu \cite{Qiu2024-2} with
Chaudhuri-Trudinger \cite{CT05} and Savin \cite{S07}. To be more precise, they proved a powerful almost Jacobi inequality to obtain the doubling inequality when $n = 4$; their method also provides a new proof for $n = 3$ and a Hessian estimate for smooth solutions satisfying a dynamic semi-convexity condition for $n \geq 5$. Fan \cite{Fan25} extended the results in \cite{Qiu2024-2, SY2025} to general $\psi$.
Besides the above listed progress, important breakthroughs are also made by
McGonagle, Song and Yuan \cite{MSY2019}, Guan and Qiu \cite{GQ19}, Shankar and Yuan \cite{SY2020, SY2025}, Mooney \cite{Mooney2021} under certain convexity assumptions. For $n \geq 5$, interior $C^2$ estimate for \eqref{2-equation} is still an open problem.

Very recently, Lu \cite{Lu2025} studied interior $C^2$ estimate for Hessian quotient equation
\begin{equation} \label{Hessian_quotient}
\bigg( \frac{\sigma_k}{\sigma_l} \bigg) (D^2 u) = \psi.
\end{equation}
Lu proved such estimate fails for
\[ 1 \leq l < k \leq n, \ \ k - l \geq 3, \]
while works for $k = n$ and
$l = n - 1$ or $l = n - 2$ by using a special concavity property of positive quotient operators proved by Guan and Sroka \cite{GS25}.
For curvature equation
\[ \bigg( \frac{\sigma_k}{\sigma_{k - 1}} \bigg) (\kappa) = \psi, \]
interior $C^2$ estimate was derived by Sheng, Urbas and Wang \cite{SUW2004}.
An important case of \eqref{Hessian_quotient} is
\begin{equation} \label{3over1}
\bigg( \frac{\sigma_3}{\sigma_1} \bigg) (D^2 u) = \psi.
\end{equation}
For $n = 3, 4$ and $\psi \equiv 1$, interior $C^2$ estimate for \eqref{3over1} was obtained by Chen, Warren and Yuan \cite{CWY2009} and Wang and Yuan \cite{WY2014} using special Lagrangian structure of the equation;
for $n = 3$ by Lu \cite{Lu} via Jacobi inequality and Legendre transform;
for $n = 3, 4$ by Zhou \cite{Zhou2024} utilizing twisted special Lagrangian structure of the equation.
The remaining cases
are to be explored.

Different from the above mentioned techniques, in this paper, we shall use Lagrange multiplier method to compute the concavity of fully nonlinear operator accurately. That is to say, for Hessian equation
\begin{equation} \label{Hessian_general}
F( D^2 u ) := f \Big( \lambda (D^2 u) \Big) = \psi,
\end{equation}
we shall find the extreme value of the quadratic form involving third order terms under the constraint
\[ \sum_j f_j u_{jji} = \psi_i \]
and the critical equation of the test function containing $\Delta u$, which leads to another constraint
\[ \sum_j u_{jji} = A_i, \]
for each fixed $i$.
When the algebraic structure of $f$ is simple, for example, when $f = \sigma_2$ or $f = \frac{\sigma_2}{\sigma_1}$, the extreme value can be computed explicitly. Then, we shall make use of this optimal value of concavity to tackle the difficult terms during the estimation. As a result, we can recover Shankar and Yuan's almost Jacobi inequality \cite{SY2025} in a different form.

Our proof of interior Hessian estimate follows the route of Shankar and Yuan \cite{SY2025}.
The first step is to establish a crucial doubling inequality for second order derivatives of the solution (The reader is referred to \cite{Qiu2024-2} for the original motivation of establishing the doubling inequality).
Our novelty is to replace almost Jacobi inequality \cite{SY2025} by Lagrange multiplier method, and analyze the exact and explicit extreme value of concavity to reach a fine estimation. This methodology may have the potential to be applied to other problems. Our second step is to use the theory of Chaudhuri and Trudinger \cite{CT05} and Savin \cite{S07} to obtain the interior estimate \eqref{interior} as Shankar and Yuan \cite{SY2025} and Fan \cite{Fan25}.

This paper is organized as follows. In Section 2, we present some preliminaries, which may
be used in the following sections. We prove the doubling inequality for $2$-convex solutions in Section 3.
The proof of Theorem \ref{thm1} is given in Section 4. In Section 5, we provide a different proof of the doubling inequality for
$2$-convex solutions to the $2$-Hessian equation \eqref{2-equation}.

As the present paper was finished, we noticed an interesting work of Lu and Sroka \cite{LS2025} where they studied the Liouville's theorem
and interior Hessian estimates for the equation
\begin{equation}
\label{ls}
\left(\frac{\sigma_2}{\sigma_1}\right) (D^2 u) = 1
\end{equation}
namely, the special case of \eqref{Hessian} as $\psi \equiv 1$.
They discovered a relation between the Hessian
quotient operator $\sigma_2/\sigma_1$ and the Hessian operator $\sigma_2$, so that results on the equation
\[
\sigma_2 (D^2 u) = \mbox{constant}
\]
can be applied to study the equation \eqref{ls}.

\vspace{4mm}

\section{Preliminaries}

Throughout the paper, we denote
\[
\sigma_{k;i_1\cdots i_l} (\lambda) = \sigma_k (\lambda)\big|_{\lambda_{i_1} = \cdots = \lambda_{i_l} = 0}
\]
for $1 \leq k \leq n$.
\begin{lemma}
\label{js-lem1}
Let $\lambda = (\lambda_1, \ldots, \lambda_n) \in \Gamma_2$ with $\lambda_1 \geq \cdots \geq \lambda_n$,
\[
f (\lambda) = \frac{\sigma_2}{\sigma_1} (\lambda) \mbox{ and } f_i = \frac{\partial f}{\partial \lambda_i}, \ i =1, \ldots, n.
\]
Then we have
\begin{equation}
\label{js-21}
f_1 \lambda_1^2 \geq \frac{2}{n^2} f^2 (\lambda).
\end{equation}
\end{lemma}
\begin{proof}
We have
\[
\begin{aligned}
f_i = \,& \frac{\sigma_1 \sigma_{1;i} - \sigma_2}{\sigma_1^2} = \frac{(\sigma_{1;i} + \lambda_i) \sigma_{1;i} - (\lambda_i \sigma_{1;i} + \sigma_{2;i})}{\sigma_1^2}\\
 = \,& \frac{\sigma_{1;i}^2 - \sigma_{2;i}}{\sigma_1^2} = \frac{\sum_{j \neq i} \lambda_j^2 + \sigma_{2;i}}{\sigma_1^2}\\
 = \,& \frac{1}{2} \frac{\sum_{j \neq i} \lambda_j^2 + \sigma_{1;i}^2}{\sigma_1^2}.
\end{aligned}
\]
Note that (See c.f. Lemma 3.1 in \cite{CW01}.)
\[
\sigma_{1;1} \lambda_1 \geq \frac{2}{n} \sigma_2.
\]
We obtain
\[
f_1 \lambda_1^2 \geq \frac{1}{2 \sigma_1^2} (\sigma_{1;1}\lambda_1)^2 \geq \frac{2}{n^2} \left(\frac{\sigma_2}{\sigma_1}\right)^2
  = \frac{2}{n^2} f^2
\]
which is \eqref{js-21}.
\end{proof}
\begin{lemma}
\label{js-lem2}
Let $\lambda = (\lambda_1, \ldots, \lambda_n) \in \Gamma_2$ with $\lambda_1 \geq \cdots \geq \lambda_n$. We have
\begin{equation}
\label{js-22}
\begin{aligned}
f_1 \leq \,& \left(\frac{n-1}{n}\right) - \frac{f}{\sigma_1},\\
\left(1-\frac{1}{\sqrt{2}}\right) - \frac{f}{\sigma_1} \leq \,& f_i \leq 2 \left(\frac{n-1}{n}\right) - \frac{f}{\sigma_1}, \ \ i \geq 2.
\end{aligned}
\end{equation}
\end{lemma}
\begin{proof}
\eqref{js-22} follows directly from Corollary 2.1 of \cite{SY2025} although in their result there is a restriction $\sigma_2 (\lambda) = 1$.
Their proof works because in our case, each
\[
f_i = \frac{\sigma_{1;i}}{\sigma_1} - \frac{\sigma_2}{\sigma_1^2}
\]
is invariant under a scaling $\lambda \rightarrow a \lambda$ for any positive constant $a$.
\end{proof}
In the proof of Theorem \ref{thm1}, we need to consider nonsmooth $2$-convex viscosity solutions to \eqref{Hessian}. For the reader's
convenience, we give the definitions as follows. For more details, the reader is referred to \cite{TW99}.
\begin{definition}
An upper semi-continuous function, $u: \Omega \rightarrow [-\infty, \infty)$, is called $k$-convex in $\Omega$ if $\sigma_k (\lambda(D^2 q)) > 0$ for all quadratic polynomials $q$ for which the difference $u - q$ has a finite local maximum in $\Omega$.
\end{definition}

\vspace{4mm}
\section{The doubling inequality}

\begin{theorem}
\label{doubling}
Let $u$ be a $2$-convex solution to \eqref{Hessian} in $B_4 (0) \subset \mathbb{R}^n$. Suppose
\begin{equation}
\label{js-condition}
\frac{\lambda_{\min} (D^2 u)}{\Delta u} \geq -c_n, \
c_n := \frac{\sqrt{3n^2+1} - n + 1}{2n}.
\end{equation}
We have
\begin{equation}
\label{js-doubling}
\sup_{B_2 (0)} \Delta u \leq C (1+ \sup_{B_{1} (0)} \Delta u).
\end{equation}
\end{theorem}
\begin{proof}
We consider the test function in $B_3(0)$,
\[
W = \rho^\alpha \exp\left\{a (x\cdot \nabla u - u) + b \frac{|\nabla u|^2}{2}\right\} \log \max\left\{\frac{\Delta u}{M_1}, \gamma\right\},
\]
where $M_1 = \sup_{B_{1} (0)} \Delta u$, $\rho (x) = 3^2- |x|^2$, $\alpha$, $a$, $b$ and $\gamma \geq 2$ are positive constants to be chosen.
Suppose $W$ achieves its maximum at $x_0 \in B_3$. We may assume $\Delta u (x_0) \geq (\gamma+1) M_1$ for otherwise we are done.
We may also assume that $D^2 u (x_0) = \lambda_i \delta_{ij}$ is diagonal and
\[
\lambda_1 \geq \cdots \geq \lambda_n.
\]
Let
\[
F^{ij} (x) = \frac{\partial (\sigma_2/\sigma_1)}{u_{ij}}(D^2 u (x)).
\]
Then $F^{ij} (x_0) = f_{i} \delta_{ij}$ is also diagonal at $x_0$ and furthermore,
\[
0 < f_1 \leq \cdots \leq f_n.
\]
By differentiating the equation \eqref{Hessian} twice, we get
\begin{equation}
\label{js-3}
F^{ij} u_{ijl} = \psi_l + \psi_u u_l, \ \ \mbox{ for each } l = 1, \ldots, n
\end{equation}
and
\begin{equation}
\label{js-4}
F^{ij} \Delta u_{ij} + \sum_{l=1}^n F^{ij, pq} u_{ijl} u_{pql} = \Delta_x \psi + 2 \nabla \psi_u \cdot \nabla u + \psi_u \Delta u + \psi_{uu} |\nabla u|^2,
\end{equation}
where
\[
F^{ij, pq} = \frac{\partial^2 \big( \frac{\sigma_2}{\sigma_1} \big)}{\partial u_{ij} \partial u_{pq}} (D^2 u).
\]
%By the concavity of $\sqrt{\sigma_2}$ we have
%\begin{equation}
%\label{js-18}
%F^{ij, pq} u_{ijl} u_{pql} \leq \frac{1}{2} (F^{ij} u_{ijl})^2 = 0
%\end{equation}
As $\Delta u (x_0) \geq (\gamma+1) M_1$, the function
\[
\alpha \ln \rho + a(x\cdot \nabla u - u) + b \frac{|\nabla u|^2}{2} + \log \log \frac{\Delta u}{M_1}
\]
also attains its local maximum at $x_0$. For simplicity, let $U := \log\Delta u - \log M_1$. We have, at $x_0$,
\begin{equation}
\label{js-1}
\frac{\Delta u_i}{U \Delta u} + \alpha \frac{\rho_i}{\rho} + a x_i u_{ii} + b u_i u_{ii} = 0, \ \ \mbox{ for } i = 1, \ldots, n
\end{equation}
and
\begin{equation}
\label{js-2}
\begin{aligned}
0 \geq & f_i \left(\frac{\Delta u_{ii}}{U \Delta u} - (1+ U)\frac{(\Delta u_i)^2}{(U \Delta u)^2} \right.\\
  & \left. + a \lambda_i + b \lambda_i^2
+ (ax_k + b u_k) u_{kii} + \alpha \frac{\rho_{ii}}{\rho} - \alpha \frac{\rho_i^2}{\rho^2}\right)\\
 \geq & f_i \left(\frac{\Delta u_{ii}}{U \Delta u} - (1+ U)\frac{(\Delta u_i)^2}{(U \Delta u)^2} + \alpha \frac{\rho_{ii}}{\rho} - \alpha \frac{\rho_i^2}{\rho^2}\right)
  + a \psi + b f_i \lambda_i^2\\
 & - C (a+b).
   % = f_i \left(\frac{\Delta u_{ii}}{\Delta u} + \frac{\rho_{ii}}{\rho} - 2\frac{\rho_i^2}{\rho^2}\right).
\end{aligned}
\end{equation}
%We first consider the case that there exists $j \geq 2$ such that $|\lambda_j| \geq \delta \Delta u$, where $\delta$ is a positive constant sufficiently small
%to be determined. By \eqref{js-4}, \eqref{js-18}, \eqref{js-1} and \eqref{js-2} we have
%\[
%0 \geq f_i \left(\alpha \frac{\rho_{ii}}{\rho} - (3\alpha^2 + \alpha) \frac{\rho_i^2}{\rho^2}\right)
% + (b-3b^2 u_i^2 - 3a^2 x_i^2) f_i \lambda_i^2.
%   % = f_i \left(\frac{\Delta u_{ii}}{\Delta u} + \frac{\rho_{ii}}{\rho} - 2\frac{\rho_i^2}{\rho^2}\right).
%\]
%Assuming $12a^2 < b < \frac{1}{12\sup_{B_1} |\nabla u|^2}$, by \eqref{SY-eq1}, we get
%\begin{equation}
%\label{js-17}
%0 \geq f_i \left(\alpha \frac{\rho_{ii}}{\rho} - (3 \alpha^2 + \alpha) \frac{\rho_i^2}{\rho^2}\right)
% + \frac{b c_0 \delta^2}{2} \Delta u^2 \sum f_i,
%\end{equation}
%where $c_0$ is positive constant depending only on $n$. Then we obtain a bound of $\rho\Delta u$.

%Thus, without loss of generality, we assume that $|\lambda_j| < \delta \Delta u$ for all $j \geq 2$.

We note that
\begin{equation} \label{js-5-q}
\begin{aligned}
- \sum_{pqrs} F^{pq, rs} u_{pqj} u_{rsj}  = & \sum\limits_{p \neq q} \frac{f_p - f_q}{\lambda_q - \lambda_p} u_{pqj}^2 - \sum\limits_{pq} \frac{\partial^2 f}{\partial \lambda_p \partial \lambda_q} u_{ppj} u_{qqj} \\
\geq  & 2 \sum\limits_{p \neq j} \frac{f_p - f_j}{\lambda_j - \lambda_p} u_{pjj}^2 - \sum\limits_{pq} \frac{\partial^2 f}{\partial \lambda_p \partial \lambda_q} u_{ppj} u_{qqj}\\
 = & \frac{1}{\Delta u} \left(2\sum\limits_{p \neq j} u_{pjj}^2 - \sum\limits_{p\neq q} u_{ppj} u_{qqj} + 2 \sum_p f_p u_{ppj} \sum_q u_{qqj} \right).
\end{aligned}
\end{equation}
By \eqref{js-1} we realize that
\begin{equation}
\label{js-6-q}
 \sum\limits_{p \neq q} u_{ppj} u_{qqj}
= (U \Delta u)^2 A_j^2 - \sum\limits_p u_{ppj}^2,
\end{equation}
where $A_j$ is defined by
\[
A_j = \alpha \frac{\rho_j}{\rho} + a x_j \lambda_{j} + b u_j \lambda_{j}, \ \ j = 1, \ldots, n.
\]
Combining \eqref{js-3}, \eqref{js-4}, \eqref{js-1}, \eqref{js-2}, \eqref{js-5-q} and \eqref{js-6-q}, we obtain
\begin{equation}
\label{js-7}
\begin{aligned}
0 \geq & \alpha f_i \left(\frac{\rho_{ii}}{\rho} - \frac{\rho_i^2}{\rho^2}\right) - (1+ U)f_i A_i^2 - U \sum_{i=1}^n A_i^2\\
  & + \frac{1}{U \Delta u^2}\sum_{i=1}^n Q_i
 + b f_i \lambda_i^2 - C (a+b) - \frac{C}{U} - \frac{C \sum_i |A_i|}{\Delta u},
% \geq & \alpha f_i \left(\frac{\rho_{ii}}{\rho} - \frac{\rho_i^2}{\rho^2}\right) - \left(3-\frac{2}{n}\right)\Delta u \sum_{i=2}^n A_i^2
%    - [1+(n-1)\delta] \Delta u A_1^2\\
%  & + \frac{1}{\Delta u}\sum_{i=1}^n Q_i
%  + b \sum_{i=1}^n f_i \lambda_i^2,
\end{aligned}
\end{equation}
where
\[
Q_i = 2 \sum\limits_{p \neq i} u_{ppi}^2 + \sum\limits_p u_{ppi}^2, \ \ i = 1, \ldots, n.
\]
For each $i = 1, \ldots, n$, we shall minimize the quantity $Q_i$ under conditions \eqref{js-3} and \eqref{js-1}.
It suffices to minimize the function $L_i$ in terms of $t_1$, \ldots, $t_n$, $\mu_1$, $\mu_2$.
\[
L_i =  2 \sum\limits_{j \neq i} t_j^2 + \sum\limits_j t_j^2
- \mu_1 \left(\sum_j f_j t_j - B_i\right) - \mu_2 \left(\sum_j t_j + (U \Delta u) A_i\right),
\]
where the quantity $B_i$ is defined by
\[
B_i = \psi_i + \psi_u u_i.
\]
To find the critical points of $L_i$, we solve the following linear equations.
\begin{equation} \label{js-8}
\begin{aligned}
& \frac{\partial L_i}{\partial t_j} = 6 t_j - \mu_1 f_j - \mu_2 = 0, \qquad j \neq i, \qquad \qquad  & (j) \\
& \frac{\partial L_i}{\partial t_i} = 2 t_i - \mu_1 f_i - \mu_2 = 0,  \qquad \qquad & (i) \\
& \sum_j f_j t_j = B_i,  \qquad \qquad & (c) \\
& \sum_j t_j = - (U \Delta u) A_i.  \qquad \qquad & (d)
\end{aligned}
\end{equation}
Taking
\[   t_j =  \frac{\mu_1}{6} f_j + \frac{\mu_2}{6}, \qquad  j \neq i  \]
and
\[   t_i = \frac{\mu_1}{2} f_i + \frac{\mu_2}{2}   \]
into equation $(c)$ and $(d)$ yields
\begin{equation} \label{js-9}
\bigg( \sum\limits_{j \neq i} f_j^2 + 3 f_i^2 \bigg) \mu_1 + \bigg( \sum\limits_{j \neq i} f_j + 3 f_i \bigg) \mu_2 = 6 B_i,
\end{equation}
and
\begin{equation} \label{js-10}
\bigg( \sum\limits_{j \neq i} f_j + 3 f_i \bigg) \mu_1 + (n + 2) \mu_2 = - 6 (U \Delta u) A_i.
\end{equation}
For each $i$, let
\[
\begin{aligned}
R_i = \,& \sum_{l \neq i} f_l^2 + 3 f_i^2 = \sum_{l=1}^n f_l^2 + 2 f_i^2,\\
S_i = \,& \sum_{l \neq i} f_l + 3 f_i = \sum_{l=1}^n f_l + 2 f_i.
\end{aligned}
\]
By \eqref{js-9} and \eqref{js-10}, we get
\begin{equation} \label{js-11}
\begin{aligned}
 \mu_1
= \frac{6 (n+2) B_i + 6 S_i (U \Delta u) A_i}{(n + 2)R_i - S_i^2}
\end{aligned}
\end{equation}
and
\begin{equation} \label{js-12}
\mu_2 = - \frac{6S_i B_i + 6 R_i (U \Delta u) A_i}{(n + 2)R_i - S_i^2}.
\end{equation}
Hence we obtain
\begin{equation} \label{js-13}
t_j =  \frac{\left[(n+2) f_j - S_i\right] B_i + (S_i f_j - R_i)(U \Delta u) A_i}{(n + 2)R_i - S_i^2}, \qquad  j \neq i
\end{equation}
and
\begin{equation} \label{js-14}
t_i =  \frac{3\left[(n+2)f_i - S_i\right] B_i + 3 (f_i S_i - R_i)(U \Delta u) A_i}{(n + 2)R_i - S_i^2}.
\end{equation}
To sum up, we obtain that the minimum of
\[
Q_i = 2 \sum\limits_{j \neq i} t_j^2 + \sum\limits_j t_j^2 = 3 \sum\limits_{j \neq i} t_j^2 + t_i^2
\]
is achieved at $t_j$, $1\leq j \leq n$ given by \eqref{js-13} and \eqref{js-14}.
By calculations, we see
\[
\begin{aligned}
\left[(n + 2)R_i - S_i^2\right]^2 \min Q_i = \,&
3 \sum_{j\neq i} \bigg\{(R_i - f_j S_i)^2 (U \Delta u)^2A_i^2 + \left[(n+2)f_j - S_i\right]^2 B_i^2\\
 & - 2 (U \Delta u) A_i B_i \left[(n+2)f_j - S_i\right](R_i - f_j S_i)\bigg\}\\
 & + 9 \bigg\{(R_i - f_i S_i)^2 (U \Delta u)^2A_i^2 + \left[(n+2)f_i - S_i\right]^2 B_i^2\\
 & - 2 (U \Delta u) A_i B_i \left[(n+2)f_i - S_i\right](R_i - f_i S_i)\bigg\}\\
 = \,& \left[(n + 2)R_i - S_i^2\right] \bigg\{3R_i(U \Delta u)^2A_i^2 + 3(n+2) B_i^2\\
 & + 6 S_i (U \Delta u)A_i B_i\bigg\}.
\end{aligned}
\]
It follows that
\[
\min Q_i = \frac{3R_i(U \Delta u)^2A_i^2 + 3(n+2) B_i^2
 + 6 S_i (U \Delta u)A_i B_i}{(n + 2)R_i - S_i^2}.
\]
Since
\[
f_l = \frac{1}{\Delta u} \left(\sigma_{1;l} - \frac{\sigma_2}{\sigma_1}\right) = \frac{1}{\Delta u} \left(\Delta u - \lambda_l - \psi \right),
\]
we have
\begin{equation}
\label{js-16}
R_i = \sum_{l=1}^n f_l^2 + 2 f_i^2 = \frac{1}{\Delta u^2}\left((n-1)\Delta u^2 + 2 \Delta u^2 f_i^2 - 2n \Delta u \psi + n \psi^2 \right)
\end{equation}
and
\begin{equation}
\label{js-17}
S_i = \sum_{l=1}^n f_l + 2 f_i = \frac{1}{\Delta u}\left((n-1)\Delta u - n \psi + 2\Delta u f_i\right).
\end{equation}

Next, from \eqref{js-16} and \eqref{js-17} we see
\[ \begin{aligned}
& (n + 2) R_i - S_i^2 \\
= & \frac{1}{(\Delta u)^2} \bigg( 3 (n - 1) (\Delta u)^2 + 2 n (\Delta u)^2 f_i^2 - 4 (n - 1)(\Delta u)^2 f_i \\
& + 4 n \Delta u f_i \psi - 6 n \psi \Delta u + 2 n \psi^2 \bigg).
\end{aligned} \]
It follows that
\begin{equation}
\label{js-18}
\frac{6 S_i (U \Delta u)|A_i B_i|}{(n + 2)R_i - S_i^2} \leq C_0 U \Delta u |A_i|
\end{equation}
for some constant $C_0$ depending only on $n$, $\|\psi\|_{C^1}$ and $\|u\|_{C^1(B_3 (0))}$.
Let
\[
\tilde{Q}_i = \frac{3 R_i}{(n+2) R_i - S_i^2}.
\]
By calculations, we have
\[
\tilde{Q}_i = \frac{3(n-1)\Delta u^2 + 6 \Delta u^2 f_i^2 - 6n \psi \Delta u + 3n \psi^2}{3(n-1) \Delta u^2 + 2n \Delta u^2 f_i^2 - 4(n-1) \Delta u^2 f_i + 4n \Delta u f_i \psi - 6n \psi \Delta u + 2n \psi^2}.
\]
Therefore,
\begin{equation}
\label{js-24}
\begin{aligned}
& \frac{\tilde{Q}_i - 1}{\Delta u}\\
 = \,& \frac{f_i}{\Delta u} \frac{4 (n-1) \Delta u^2 - (2n-6) \Delta u^2 f_i -4n \psi \Delta u + \frac{n \psi^2}{f_i}}{3(n-1) \Delta u^2 + 2n \Delta u^2 f_i^2 - 4(n-1) \Delta u^2 f_i + 4n \Delta u f_i \psi - 6n \psi \Delta u + 2n \psi^2}\\
 = \,& \frac{f_i}{\Delta u} \left\{1+ \frac{\Delta u^2\left[(n-1)+(2n+2)f_i - 2n f_i^2\right] - 4n \psi \Delta u f_i + 2n \psi \Delta u + \frac{n \psi^2}{f_i} - 2n \psi^2}{3(n-1) \Delta u^2 + 2n \Delta u^2 f_i^2 - 4(n-1) \Delta u^2 f_i + 4n \psi \Delta u f_i - 6n \psi \Delta u + 2n \psi^2}\right\}.
\end{aligned}
\end{equation}
Since $\lambda_n \geq -c_n \Delta u$, we have
\[
f_i \leq f_n \leq \frac{n+1 + \sqrt{3n^2+1}}{2n} - \frac{\psi}{\Delta u} \leq \frac{n+1 + \sqrt{3n^2+1}}{2n} \mbox{ for each } i.
\]
Thus, we have
\begin{equation}
\label{js-25}
(n-1)+(2n+2)f_i - 2n f_i^2 \geq 0.
\end{equation}
We first consider the case $-\frac{c_n}{2} \Delta u \geq \lambda_n \geq - c_n \Delta u$. From \eqref{js-24} and \eqref{js-25} we derive
\[
\frac{\tilde{Q}_i - 1}{\Delta u} \geq \frac{f_i}{\Delta u} \left(1 - C\frac{1}{\Delta u}\right).
\]
Consequently, by \eqref{js-7} and \eqref{js-18}, we obtain
\begin{equation}
\label{js-15}
\begin{aligned}
0 \geq \,& \alpha f_i \left(\frac{\rho_{ii}}{\rho} - \frac{\rho_i^2}{\rho^2}\right) - f_i A_i^2 - \frac{C U}{\Delta u} f_i A_i^2 + b f_i \lambda_i^2\\
 & - \sum_{i=1}^n\frac{C |A_i|}{\Delta u}
  - C (a+b) - \frac{C}{U} \\
 \geq \,& \alpha f_i \left(\frac{\rho_{ii}}{\rho} - \frac{\rho_i^2}{\rho^2}\right) - \left(1+\frac{C U}{\Delta u}\right) f_i \left(\frac{\alpha \rho_i}{\rho} + a x_i \lambda_i + b u_i \lambda_i\right)^2 + b f_i \lambda_i^2\\
 & - \frac{C\alpha}{\rho \Delta u} - C (a+b) - \frac{C}{U} \\
 \geq \,& \left[b- C(a^2+b^2)\right] f_i \lambda_i^2 - \frac{C(\alpha+\alpha^2)}{\rho^2} - C (a+b)\\
 \geq \,& \frac{b}{2} f_i \lambda_i^2 - \frac{C(\alpha+\alpha^2)}{\rho^2} - C (a+b)
\end{aligned}
\end{equation}
by assuming $a^2 \ll b \ll 1$. By \eqref{js-22} and \eqref{js-15} we get
\[
0 \geq \frac{bc_0 c_n^2}{8} \Delta u^2 - \frac{C(\alpha+\alpha^2)}{\rho^2} - C (a+b)
\]
and a bound of $\rho \Delta u$ is derived, where $c_0$ is a known positive constant in \eqref{js-22}.

Next, we deal with the case $\lambda_n \geq -\frac{c_n}{2} \Delta u$. In this case, \eqref{js-25} can be improved to
\[
(n-1)+(2n+2)f_i - 2n f_i^2 \geq \epsilon_0
\]
for some positive constant $\epsilon_0$ depending only on $n$.
Thus, if $\Delta u$ is large enough, there exists
a positive constant $\epsilon_1$ depending only on $n$ such that
\begin{equation}
\label{js-23}
\frac{\tilde{Q}_i - 1}{\Delta u} \geq \frac{f_i}{\Delta u} (1 + \epsilon_1).
\end{equation}
By \eqref{js-7}, \eqref{js-18} and \eqref{js-23}, we have (provided $U$ is sufficiently large)
\begin{equation}
\label{js-26}
\begin{aligned}
0 \geq \,& \alpha f_i \left(\frac{\rho_{ii}}{\rho} - \frac{\rho_i^2}{\rho^2}\right) - f_i A_i^2 + \epsilon_1 U f_i A_i^2 + b f_i \lambda_i^2\\
 & - \sum_{i=1}^n\frac{C |A_i|}{\Delta u}
  - C (a+b) - \frac{C}{U} \\
 \geq \,& \alpha f_i \left(\frac{\rho_{ii}}{\rho} - \frac{\rho_i^2}{\rho^2}\right) + \frac{\epsilon_1}{2} U f_i A_i^2+ b f_i \lambda_i^2\\
 & - \frac{C\alpha}{\rho \Delta u} - C (a+b) - \frac{C}{U} \\
 \geq \,& \frac{\epsilon_1}{2} U f_i A_i^2+ b f_i \lambda_i^2 - C (a+b) - \frac{C\alpha}{\rho^2}.
\end{aligned}
\end{equation}
Note that $x_0 \in B_3(0) - B_1 (0)$. We consider two cases.

\textbf{Case 1.} $x_1^2 \geq \frac{1}{n}$. We may further assume $a > 2\sqrt{n} \sup |\nabla u| b$ to obtain that
\[
\begin{aligned}
A_1^2 = \,& \left(-2 \alpha \frac{x_1}{\rho} + b u_1 \lambda_{1} + a x_1 \lambda_{1}\right)^2\\
  \geq \,& \frac{a^2}{8n} \lambda_1^2 - C\frac{\alpha^2}{\rho^2} \geq \frac{a^2}{16n} \lambda_1^2
\end{aligned}
\]
since $\lambda_1$ can be sufficiently large.
By \eqref{js-21} and \eqref{js-26}, we get
\[
0 \geq - \frac{C\alpha}{\rho^2} + \frac{\epsilon_1}{2} U \frac{a^2}{8n^3} - C (a+b).
\]
We then get a bound of $\rho^2 U$.

\textbf{Case 2.} $x_j^2 \geq \frac{1}{n}$ for some $j \geq 2$.
If
\[
\frac{\alpha}{\rho} \leq \frac{|b u_j \lambda_{j} + a x_j \lambda_{j}|}{|x_j|},
\]
by \eqref{js-26} and fixing $a^2 \ll b \ll a \ll 1 \leq \alpha$, we get
\[
\begin{aligned}
0 \geq \,& - \frac{C\alpha}{\rho^2} + b f_j \lambda_j^2 - C (a+b)\\
\geq \,& b f_j \lambda_j^2 - \frac{C\alpha}{\rho^2}
 \geq - C \left[(a^2 + b^2) \lambda_j^2\right] + b f_j \lambda_j^2 > 0.
\end{aligned}
\]
Here we have used the fact that
\[
f_j = \frac{\sigma_{1;j}}{\sigma_1} - \frac{\psi}{\Delta u} \geq c_1 > 0,\ \ j \geq 2
\]
for some constant $c_1$ depending only on $n$.

Otherwise, we suppose
\[
\frac{\alpha}{\rho} \geq \frac{|b u_j \lambda_{j} + a x_j \lambda_{j}|}{|x_j|},
\]
We find
\[
A_j^2 = \left(-2 \alpha \frac{x_j}{\rho} + b u_j \lambda_{j} + a x_j \lambda_{j}\right)^2
  \geq \frac{\alpha^2}{n\rho^2}.
\]
We then get a bound of $\rho U$ again by \eqref{js-26}.

\end{proof}
\begin{remark}
\label{js-rmk}
Since $\Delta u$ can be sufficiently large at the maximum point $x_0$, a semi-convex solution $u$
satisfies the condition \eqref{js-condition} naturally at $x_0$. Noticing that we only need to do calculations at
$x_0$ in the proof of Theorem \ref{doubling}, we find the doubling inequality \eqref{js-doubling} holds for semi-convex
$2$-convex solutions in general dimensions.
\end{remark}

\vspace{4mm}

\section{Proof of Theorem \ref{thm1} and Theorem \ref{thm2}}

The rest of proof is similar to that in \cite{SY2025} and \cite{Fan25}. For completeness, we provide a sketch here.
In this section, we assume that $n=3$ or the solution $u$ is semi-convex in dimensions $n\geq 4$. We assume
$u$ is defined in $B_4 (0)$ by scaling $16 u (x/4)$ and prove that $|D^2 u (0)|$ is controlled by $\|u\|_{C^1 (B_4 (0))}$. Otherwise, there exists
a sequence of $2$-convex smooth solutions $u_k$ of \eqref{Hessian} in $B_4 (0)$ such that $\|u_k\|_{C^1 (B_3(0))} \leq M$ independent of $k$, but
$|D^2 u_k (0)| \rightarrow \infty$. By Arzela-Ascoli's theorem, up to a subsequence, $u_k$ uniformly converges to a Lipschitz function $u$ in $B_3(0)$.
By the closeness of viscosity solutions, $u$ is a $2$-convex viscosity solution of \eqref{Hessian}. (c.f. \cite{CC95}, see Lemma 6.1 of \cite{SY2025} also.)

If $u$ is semi-convex, by the classic Alexandrov theorem for convex functions (see Theorem 6.9 in \cite{EG92}), $u$ satisfies the Alexandrov type regularity,
i.e., $u$ is twice differentiable almost everywhere.

In dimension three, since $u$ is $2$-convex in $\mathbb{R}^3$, the Alexandrov type regularity of $u$
follows immediately from Theorem 1.1 of Chaudhuri-Trudinger \cite{CT05}.

Fix a point $y \in B_{1/3} (0)$ such that $u$ is twice differentiable at $y$. Let $Q(x)$ be
the quadratic polynomial such that
\[
|u (x) - Q (x)| = o (|x-y|^2).
\]

Let $v_k = u_k - Q$ and near $y$, and
\[
\bar{v}_k (\bar{x}) = \frac{1}{r^2} v_k (r \bar{x} + y), \ \ \bar{x} \in B_1 (0)
\]
for small $r > 0$. We see
\[
\|\bar{v}_k\|_{L^\infty (B_1(0))} \leq \frac{\|u_k (r \bar{x} + y) - u (r \bar{x} + y)\|_{L^\infty (B_1(0))}}{r^2} + \sigma(r),
\]
where $\sigma(r) = \frac{o(r^2)}{r^2}$. Furthermore, $\bar{v}_k$ satisfies the equation
\[
\begin{aligned}
\frac{\sigma_2}{\sigma_1} \left(D^2 \bar{v}_k (\bar{x}) + D^2 \frac{Q (r \bar{x} + y)}{r^2}\right)
 = \,& \psi (r \bar{x} + y, u_k (r \bar{x} + y))\\
 = \,& \psi (r \bar{x} + y, r^2 \bar{v}_k (\bar{x}) + Q(r \bar{x} + y)).
\end{aligned}
\]
To proceed we define
\[
G (M, z, \bar{x}) = \frac{\sigma_2}{\sigma_1} (M + D^2 Q) - \frac{\sigma_2}{\sigma_1} (D^2 Q) - \psi (r\bar{x}+y, r^2 z + Q) + \psi(r\bar{x} + y, Q)
\]
for $(M, z, \bar{x}) \in \mathbb{S}^{n\times n} \times \mathbb{R} \times B_1 (0)$. It is easy to check that $\bar{v}_k$ solves the equation
\[
G (D^2 \bar{v}_k, \bar{v}_k, \bar{x}) = \psi(r\bar{x} + y, Q) - \psi (y, Q(y))
\]
since $D^2 Q \equiv D^2 u(y)$ and $Q (y) = u (y)$. We see $G$ and $\bar{v}_k$ satisfies the conditions of Theorem 5.1 in \cite{Fan25} which is a
generalization of Savin's small perturbation theorem (\cite{S07}, Theorem 1.3) by choosing $r = \rho$ sufficiently small.
It follows that $\|\bar{v}_k\|_{C^{2,\alpha} (B_{1/2}(0))} \leq C$ independent of $k$. Thus,
\[
\Delta u_k \leq C \mbox{ in } B_{\rho/2} (y) \mbox{ uniformly in } k.
\]
Finally, by Theorem \ref{doubling} and Remark \ref{js-rmk}, we see
\[
\begin{aligned}
\sup_{B_1 (y)} \Delta u_k \leq \,& C_1\sup_{B_{1/2} (y)} \Delta u_k \leq \cdots \leq C_j \cdots C_1 \sup_{B_{1/2^j} (y)} \Delta u_k\\
 \leq \,& C_j \cdots C_1 \sup_{B_{\rho/2} (y)} \Delta u_k \leq C,
\end{aligned}
\]
where $j$ is chosen such that $\frac{1}{2^j} \leq \frac{\rho}{2}$. This contradicts the assumption $|D^2 u_k (0)| \rightarrow \infty$.
Theorem \ref{thm1} and \ref{thm2} are proved.

\vspace{4mm}

\section{Appendix}

In this section, we give a different proof of the doubling inequality for $2$-convex solutions to the $2$-Hessian equation in dimension four.
Let $u$ be a $2$-convex solution to the $2$-Hessian equation
\begin{equation}
\label{Hessian'}
\sigma_2 (D^2 u) = \psi (x, u, \nabla u)
\end{equation}
in $B_4 (0)$. We provide another proof of the following doubling inequality.
\begin{theorem}
\label{doubling-1}
Suppose
\begin{equation}
\label{js-condition-1}
\frac{\lambda_{\min} (D^2 u)}{\Delta u} \geq -c_n, \
c_n := \frac{\sqrt{3n^2+1} - n + 1}{2n}.
\end{equation}
We have
\begin{equation}
\label{js-doubling-1}
\sup_{B_2 (0)} \Delta u \leq C (1+ \sup_{B_{1} (0)} \Delta u).
\end{equation}
\end{theorem}
We consider the test function in $B_3(0)$,
\[
W = \rho^\alpha \exp\left\{a (x\cdot \nabla u - u) + b \frac{|\nabla u|^2}{2}\right\} \max\left\{\log \frac{\Delta u}{M_1}, \gamma\right\},
\]
where $M_1 = \sup_{B_{1} (0)} \Delta u$, $\rho (x) = 3^2- |x|^2$, $\alpha$, $a$, $b$ and $\gamma \geq 2$ are positive constants to be chosen.
Suppose $W$ achieves its maximum at $x_0 \in B_3$. We may assume $\Delta u (x_0) \geq e^{\gamma+1} M_1$ for otherwise we are done.
We may also assume that $D^2 u (x_0) = \lambda_i \delta_{ij}$ is diagonal and
\[
\lambda_1 \geq \cdots \geq \lambda_n.
\]
Let
\[
F^{ij} (x) = \frac{\partial \sigma_2}{\partial u_{ij}}(D^2 u (x)).
\]
Then $F^{ij} (x_0) = f_{i} \delta_{ij}$ is also diagonal at $x_0$ and furthermore,
\[
0 < f_1 \leq \cdots \leq f_n.
\]
We note that
\[
f_i = \sum_{l\neq i} \lambda_l = \Delta u - \lambda_i \mbox{ for each }i.
\]
By differentiating the equation \eqref{Hessian'} twice, we get
\begin{equation}
\label{js-3'}
F^{ij} u_{ijl} = \psi_l + \psi_u u_l + \psi_{u_s} u_{sl}, \ \ \mbox{ for each } l = 1, \ldots, n
\end{equation}
and
\begin{equation}
\label{js-4'}
F^{ij} \Delta u_{ij} + \sum_{l=1}^n F^{ij, pq} u_{ijl} u_{pql}\geq - C \Delta u^2,
\end{equation}
where
\[
F^{ij, pq} = \frac{\partial^2 \sigma_2}{\partial u_{ij} \partial u_{pq}} (D^2 u).
\]
As $\Delta u (x_0) \geq e^{\gamma+1} M_1$, the function
\[
\alpha \ln \rho + a(x\cdot \nabla u - u) + b \frac{|\nabla u|^2}{2} + \log \log \frac{\Delta u}{M_1}
\]
also attains its local maximum at $x_0$. For simplicity, let $U := \log\Delta u - \log M_1$. We have, at $x_0$,
\begin{equation}
\label{js-1'}
\frac{\Delta u_i}{U \Delta u} + \alpha \frac{\rho_i}{\rho} + a x_i u_{ii} + b u_i u_{ii} = 0, \ \ \mbox{ for } i = 1, \ldots, n
\end{equation}
and
\begin{equation}
\label{js-2'}
\begin{aligned}
0 \geq & f_i \left(\frac{\Delta u_{ii}}{U \Delta u} - (1+ U)\frac{(\Delta u_i)^2}{(U \Delta u)^2} \right.\\
  & \left. + a \lambda_i + b \lambda_i^2
+ (ax_k + b u_k) u_{kii} + \alpha \frac{\rho_{ii}}{\rho} - \alpha \frac{\rho_i^2}{\rho^2}\right)\\
\geq & f_i \left(\frac{\Delta u_{ii}}{U \Delta u} - (1+U)\frac{(\Delta u_i)^2}{(U \Delta u)^2} + \alpha \frac{\rho_{ii}}{\rho} - \alpha \frac{\rho_i^2}{\rho^2}\right)
  + 2a + b f_i \lambda_i^2\\
  & - C (a+b) \Delta u.
   % = f_i \left(\frac{\Delta u_{ii}}{\Delta u} + \frac{\rho_{ii}}{\rho} - 2\frac{\rho_i^2}{\rho^2}\right).
\end{aligned}
\end{equation}
%We first consider the case that there exists $j \geq 2$ such that $|\lambda_j| \geq \delta \Delta u$, where $\delta$ is a positive constant sufficiently small
%to be determined later. By \eqref{js-4}, \eqref{js-18}, \eqref{js-1} and \eqref{js-2} we have
%\[
%0 \geq f_i \left[\alpha \frac{\rho_{ii}}{\rho} - \left(\alpha + 3\alpha^2 (1+\log U)\right) \frac{\rho_i^2}{\rho^2}\right]
% + (b-3b^2 u_i^2 - 3a^2 x_i^2) f_i \lambda_i^2.
%   % = f_i \left(\frac{\Delta u_{ii}}{\Delta u} + \frac{\rho_{ii}}{\rho} - 2\frac{\rho_i^2}{\rho^2}\right).
%\]
%Assuming $12a^2 < b < \frac{1}{12\sup_{B_1} |\nabla u|^2}$, by \eqref{SY-eq1}, we get
%\begin{equation}
%\label{js-17}
%0 \geq f_i \left(\alpha \frac{\rho_{ii}}{\rho} - (3 \alpha^2 + \alpha) \frac{\rho_i^2}{\rho^2}\right)
% + \frac{b c_0 \delta^2}{2} \Delta u^2 \sum f_i,
%\end{equation}
%where $c_0$ is positive constant depending only on $n$. Then we obtain a bound of $\rho\Delta u$.
%
%Thus, without loss of generality, we assume that $|\lambda_j| < \delta \Delta u$ for all $j \geq 2$.

We note that
\begin{equation} \label{js-5}
\begin{aligned}
- \sum_{pqrs} F^{pq, rs} u_{pqj} u_{rsj}  = & \sum\limits_{p \neq q} \frac{f_p - f_q}{\lambda_q - \lambda_p} u_{pqj}^2 - \sum\limits_{pq} \frac{\partial^2 f}{\partial \lambda_p \partial \lambda_q} u_{ppj} u_{qqj} \\
\geq  & 2 \sum\limits_{p \neq j} \frac{f_p - f_j}{\lambda_j - \lambda_p} u_{pjj}^2 - \sum\limits_{pq} \frac{\partial^2 f}{\partial \lambda_p \partial \lambda_q} u_{ppj} u_{qqj}\\
 = & 2 \sum\limits_{p \neq j} u_{pjj}^2 - \sum\limits_{p\neq q} u_{ppj} u_{qqj}.
\end{aligned}
\end{equation}
By \eqref{js-1} we realize that
\begin{equation}
\label{js-6}
 \sum\limits_{p \neq q} u_{ppj} u_{qqj}
= (U \Delta u)^2 A_j^2 - \sum\limits_p u_{ppj}^2,
\end{equation}
where $A_j$ is defined by
\[
A_j = \alpha \frac{\rho_j}{\rho} + a x_j \lambda_{j} + b u_j \lambda_{j}, \ \ j = 1, \ldots, n.
\]
Combining \eqref{js-4}, \eqref{js-1}, \eqref{js-2}, \eqref{js-5} and \eqref{js-6}, we obtain
\begin{equation}
\label{js-7'}
\begin{aligned}
0 \geq & \alpha f_i \left(\frac{\rho_{ii}}{\rho} - \frac{\rho_i^2}{\rho^2}\right) - (1+U)f_i A_i^2 - U \Delta u \sum_{i=1}^n A_i^2
  + \frac{1}{U \Delta u}\sum_{i=1}^n Q_i\\
 & + b f_i \lambda_i^2 - C (a+b) \Delta u - \frac{C \Delta u}{U}\\
% \geq & \alpha f_i \left(\frac{\rho_{ii}}{\rho} - \frac{\rho_i^2}{\rho^2}\right) - \left(3-\frac{2}{n}\right)\Delta u \sum_{i=2}^n A_i^2
%    - [1+(n-1)\delta] \Delta u A_1^2\\
%  & + \frac{1}{\Delta u}\sum_{i=1}^n Q_i
%  + b \sum_{i=1}^n f_i \lambda_i^2,
\end{aligned}
\end{equation}
where
\[
Q_i = 2 \sum\limits_{p \neq i} u_{pii}^2 + \sum\limits_p u_{ppi}^2, \ \ i = 1, \ldots, n.
\]
%and we have used the fact
%\[
%f_1 = \sum_{i=2}^n \lambda_i \leq (n-1) \delta
%\]
%in the last inequality of \eqref{js-7}.
For each $i = 1, \ldots, n$, we shall minimize the quantity $Q_i$ under conditions \eqref{js-3} and \eqref{js-1}.
It suffices to minimize the function $L_i$ in terms of $t_1$, \ldots, $t_n$, $\mu_1$, $\mu_2$.
\[
L_i =  2 \sum\limits_{j \neq i} t_j^2 + \sum\limits_j t_j^2
- \mu_1 \left(\sum_j f_j t_j - B_i\right) - \mu_2 \left(\sum_j t_j + (U \Delta u) A_i\right),
\]
where
\[
B_i = \psi_i + \psi_u u_i + \psi_{u_s} u_{si}.
\]
To find the critical points of $L_i$, we solve the following linear equations.
\begin{equation} \label{js-8'}
\begin{aligned}
& \frac{\partial L_i}{\partial t_j} = 6 t_j - \mu_1 f_j - \mu_2 = 0, \qquad j \neq i, \qquad \qquad  & (j) \\
& \frac{\partial L_i}{\partial t_i} = 2 t_i - \mu_1 f_i - \mu_2 = 0,  \qquad \qquad & (i) \\
& \sum_j f_j t_j = B_i,  \qquad \qquad & (c) \\
& \sum_j t_j = - (U \Delta u) A_i.  \qquad \qquad & (d)
\end{aligned}
\end{equation}
As in Section 3, by calculations, we obtain
\begin{equation}
\label{js-11'}
\min Q_i = \frac{3R_i(U \Delta u)^2A_i^2 + 3(n+2) B_i^2
 + 6 S_i (U \Delta u)A_i B_i}{(n + 2)R_i - S_i^2},
\end{equation}
where
\[
\begin{aligned}
R_i = \,& \sum_{l \neq i} f_l^2 + 3 f_i^2 = \sum_{l=1}^n f_l^2 + 2 f_i^2,\\
S_i = \,& \sum_{l \neq i} f_l + 3 f_i = \sum_{l=1}^n f_l + 2 f_i.
\end{aligned}
\]
For the $2$-Hessian equation, we see
\[
f_j = \Delta u - \lambda_j, \mbox{ for } j = 1,\ldots, n.
\]
Let
\[
\tilde{Q}_i = \frac{3 R_i}{(n+2) R_i - S_i^2}.
\]
By calculations, we have
\begin{equation}
\label{js-9'}
\begin{aligned}
\tilde{Q}_i - 1 = \,& \frac{4(n-1)\Delta u f_i - 2 (n-3) f_i^2 + 2 (n-1)}{3(n-1) \Delta u^2 - 4(n-1) \Delta u f_i + 2n f_i^2 - 2(n+2)}\\
 = \,& \frac{f_i}{\Delta u}\frac{4(n-1)\Delta u - 2 (n-3) f_i + \frac{2 (n-1)}{f_i}}{3(n-1) \Delta u - 4(n-1) f_i + 2n \frac{f_i^2}{\Delta u}
  - \frac{2(n+2)}{\Delta u}}\\
  = \,& \frac{f_i}{\Delta u} \left(1 + \frac{(n-1)\Delta u + 2 (n+1) f_i - 2n \frac{f_i^2}{\Delta u} + \frac{2 (n-1)}{f_i} + \frac{2(n+2)}{\Delta u}}{3(n-1) \Delta u - 4(n-1) f_i + 2n \frac{f_i^2}{\Delta u}
  - \frac{2(n+2)}{\Delta u}}\right).
\end{aligned}
\end{equation}
Since $\lambda_n \geq -c_n \Delta u$, we have $f_i \leq f_n \leq (1+c_n) \Delta u$
\begin{equation}
\label{js-10'}
(n-1)\Delta u + 2 (n+1) f_i - 2n \frac{f_i^2}{\Delta u} \geq 0
\end{equation}
Next, we see
\begin{equation}
\label{js-12'}
\frac{6 S_i (U \Delta u)|A_i B_i|}{(n + 2)R_i - S_i^2} \leq C_0 U \Delta u |A_i|
\end{equation}
for some constant $C_0$ depending only on $n$, $\|\psi\|_{C^1}$ and $\|u\|_{C^1(B_3 (0))}$.

We first consider the case $-\frac{c_n}{2} \Delta u \geq \lambda_n \geq - c_n \Delta u$.
By \eqref{js-7'}, \eqref{js-11'}, \eqref{js-9'}, \eqref{js-10'} and \eqref{js-12'}, we obtain
\begin{equation}
\label{js-13'}
\begin{aligned}
0 \geq & \alpha f_i \left(\frac{\rho_{ii}}{\rho} - \frac{\rho_i^2}{\rho^2}\right) - (1+U)f_i A_i^2 - U \Delta u \sum_{i=1}^n A_i^2
  + \frac{1}{U \Delta u}\sum_{i=1}^n Q_i\\
 & + b f_i \lambda_i^2 - C (a+b) \Delta u - \frac{C \Delta u}{U}\\
\geq \,& \alpha f_i \left(\frac{\rho_{ii}}{\rho} - \frac{\rho_i^2}{\rho^2}\right) - f_i \left(\frac{\alpha \rho_i}{\rho} + a x_i \lambda_i + b u_i \lambda_i\right)^2 + b f_i \lambda_i^2\\
 & - \frac{C\alpha}{\rho} - C (a+b) \Delta u - \frac{C \Delta u}{U}\\
 \geq \,& \frac{b}{2} f_i \lambda_i^2 - \frac{C(\alpha+\alpha^2)}{\rho^2}\left(1+\sum f_i\right) - C (a+b) \Delta u\\
 \geq \,& \frac{bc_n^2}{8n} \Delta u^2 \sum f_i - \frac{C(\alpha+\alpha^2)}{\rho^2}\left(1+\sum f_i\right) - C (a+b) \Delta u
\end{aligned}
\end{equation}
by assuming $a^2 \ll b \ll 1$. Thus, a bound of $\rho\Delta u$ is derived as in Section 3.

For the case $\lambda_n \geq -\frac{c_n}{2} \Delta u$, we have
\[
(n-1)\Delta u + 2 (n+1) f_i - 2n \frac{f_i^2}{\Delta u} \geq \epsilon_0 \Delta u
\]
and as in Section 3,
\begin{equation}
\label{js-26'}
\begin{aligned}
0 \geq \,& \alpha f_i \left(\frac{\rho_{ii}}{\rho} - \frac{\rho_i^2}{\rho^2}\right) - f_i A_i^2 + \epsilon_1 U f_i A_i^2 + b f_i \lambda_i^2\\
 & - \sum_{i=1}^n C_0 |A_i|
  - C (a+b) \Delta u - \frac{C \Delta u}{U}\\
 \geq \,& \alpha f_i \left(\frac{\rho_{ii}}{\rho} - \frac{\rho_i^2}{\rho^2}\right) + \frac{\epsilon_1}{2} U f_i A_i^2+ b f_i \lambda_i^2\\
 & - \frac{C\alpha}{\rho} - C (a+b) \Delta u - \frac{C \Delta u}{U}\\
 \geq \,& \frac{\epsilon_1}{2} U f_i A_i^2+ b f_i \lambda_i^2 - C (a+b) \Delta u - \frac{C\alpha}{\rho^2} \left(1+\sum f_i\right).
\end{aligned}
\end{equation}
We note that the maximum of $W$ is achieved at a point in $B_2-B_1$. Thus, there exists an index
$1\leq j \leq n$ such that $x_j^2 \geq \frac{1}{n}$. Finally we consider two cases $x_1^2 \geq \frac{1}{n}$
and $x_j^2 \geq \frac{1}{n}$ for some $j \geq 2$ to obtain a bound of $\rho U$ by choosing $a^2 \ll b \ll a \ll 1 \leq \alpha$
as in Section 3 with only a few modifications.

\end{document}